\documentclass{amsart}
\usepackage{amsfonts,amssymb,amsbsy,amsmath,amsthm}
\usepackage{epsfig,float,bm}
\usepackage{etex,pictex,graphicx,hyperref,relsize,esvect,cancel,ifthen}
\usepackage[usenames]{xcolor}
\usepackage{hyperref}
\usepackage{fontawesome5}
\usepackage{caption}
\newtheorem{theorem}{Theorem}

\newtheorem{corollary}[theorem]{Corollary}

\theoremstyle{definition}

\newtheorem{remark}[theorem]{Remark}

\newcommand{\calA}{\mathcal{A}}

\begin{document}

\title[Generalizations of the Squircle-Lemniscate Relation]{Generalizations of the Squircle-Lemniscate Relation and Keplerian Dynamics}
\author{Zbigniew Fiedorowicz, Muthu Veerappan Ramalingam}
\address{\hbox{\parbox{150pt}{Department of Mathematics,\\
The Ohio State University\\
Columbus, Ohio 43235 -1174, USA}\hspace{80pt}
\parbox{\linewidth}{Aranthangi, Tamil Nadu, India}}
}

\begin{abstract}
This paper establishes a generalized relationship between the arc length of sinusoidal spirals $r^n = \cos(n\theta)$ and the area of Lam\'e curves defined by $x^{2n} + y^{2n} = 1$. Building on our previous work connecting the lemniscate to the squircle, we prove an integral identity relating these two curves for any positive integer $n$, which we further generalize to arbitrary positive real exponents and general superellipses. We further extend this correspondence to a geometric relationship between radial sectors of the Lam\'e curve and arc lengths of the spiral, providing a physical interpretation where Keplerian motion on the Lam\'e curve corresponds to uniform motion on the spiral. Additionally, we derive an explicit central force law for Keplerian motion along the Lam\'e curve. Finally, we introduce ``policles''—a new class of curves generalizing the squircle—and demonstrate a direct geometric mapping between their sectors and the arc lengths of sinusoidal spirals.
\end{abstract}

\maketitle

\section{Introduction}

In \cite{Fiedorowicz-Ramalingam} we showed there is a remarkable relation between the lemniscate
\[{(x^2+y^2)^2=x^2-y^2}\]
and the Lam\'e curve $x^4+y^4=1$, popularly known as the squircle, and we provided
an elementary proof of this relation. We also posed the question whether this result could be generalized
to an analogous relation between Lam\'e curves $x^{2n}+y^{2n}=1$ for $n>2$ and some hypothetical analogs
of the lemniscate. In this paper we answer this question and provide generalizations of our result in \cite{Fiedorowicz-Ramalingam}. 

The starting point in that paper was the following integral equality due to Levin \cite{alevin}
\begin{equation}
\int_0^1\frac{\,dr}{\sqrt{1-r^4}}=\sqrt{2}\int_0^1\sqrt[4]{1-x^4}\,dx,\label{alevin_eqn}
\end{equation}
which implies that the perimeter of the lemniscate is $\sqrt{2}$ times the area of the squircle. We then generalized
this to a relation between areas of radial sectors of the squircle and lengths of geometrically related arcs of the
lemniscate.

We also offered a physical interpretation of this result: Keplerian motion of a particle around the squircle
is geometrically related to uniform motion of a particle around the lemniscate. Here Keplerian motion means the
particle sweeps out radial sectorial area at a constant rate, as in Kepler's second law of planetary motion.

In this paper we begin by generalizing equation \eqref{alevin_eqn} to the following integral equality
\begin{equation}
\int_0^1\frac{\,dr}{\sqrt{1-r^{2n}}}=2^{\frac{1}{n}}\int_0^1\sqrt[2n]{1-x^{2n}}\,dx.\label{eqn2}
\end{equation}
This equality implies that the perimeter of the sinusoidal spiral \cite{Wikipedia:sinusoidal_spiral}, given by the polar equation $r^n=\cos(n\theta)$
is $2^{\frac{1}{n}-1}n$ times the area of the Lam\'e curve\footnote{G. Lam\'e initiated the study of these curves in 1818  \cite[p. 105]{lame}, hence
the nomenclature.} ${x^{2n}+y^{2n}=1}$, \cite{Wikipedia:superellipse}.
We then generalize the main result of \cite{Fiedorowicz-Ramalingam} relating radial sector areas of the Lam\'e curves to lengths of geometrically
related arcs of sinusoidal spirals. We then offer a physical interpretation of this result as relating Keplerian motion
along Lam\'e curves to uniform motion along sinusoidal spirals, which we illustrate graphically in the case $n=3$. We
also provide a central force interpretation of Keplerian motion on Lam\'e curves. 
Finally we discuss an alternative
generalization, which relates sinusoidal spirals to different generalizations of the squircle, policles whose polar equations
have the form
\begin{equation}
r^4=\frac{n\sin^2(n\theta)}{1-\cos^{2n}(n\theta)}
\end{equation}

Note that many integral equations discussed in this paper are valid for positive real values of $n$, and have geometric
interpretations provided the corresponding curves are restricted to portions contained within the first quadrant. However
in order to provide nicer global geometric interpretations, we will find it convenient to assume in what follows that $n$ is a
positive integer.

We would like to take this opportunity to acknowledge the valuable assistance of Google Gemini AI \cite{gemini}
which provided some crucial insights in the development of this paper. We also thank David Duncan for his encouragement to pursue this line of research.

\bigskip

\section{Sinusoidal Spirals}

A \textit{sinusoidal spiral} is a curve given by the polar equation $r^n=\cos(n\theta)$, c.f. \cite{Wikipedia:sinusoidal_spiral}, \cite{mathcurve.com}, \cite{lawrence}, \cite{maclaurin}.
While sinusoidal spirals have been studied for all rational values of $n$, both positive and
negative, as noted above we are only interested in the case when $n$ is a positive integer.  In this case the graph resembles an $n$-leaf clover and is sometimes referred to by this name e.g. \cite{thyde}, \cite{cox}. Figure \ref{fig:sinusoidal_spiral-n=5} below illustrates the case for $n=5$.
\begin{figure}[H]
\centering
\includegraphics[width=150pt]{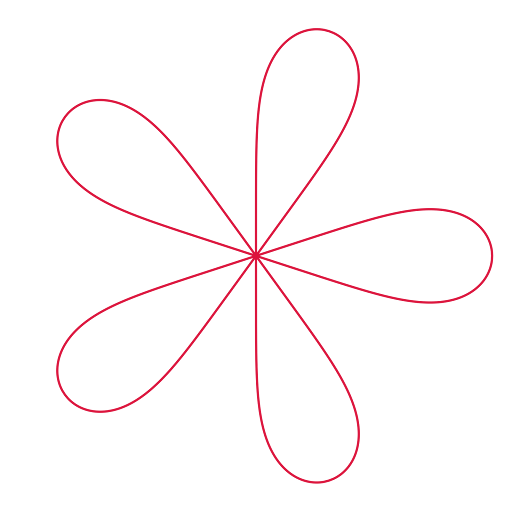}
\caption{The Sinusoidal Spiral $r^5=\cos(5\theta)$}
\label{fig:sinusoidal_spiral-n=5}
\end{figure}
When $n=1$ the sinusoidal spiral is just the circle with cartesian equation $\left(x-\frac{1}{2}\right)^2+y^2=\frac{1}{4}$
and when $n=2$, the sinusoidal spiral is just the lemniscate with cartesian equation $(x^2+y^2)^2=x^2-y^2$. The rotation
$\theta\mapsto\theta-\frac{\pi}{2n}$ converts the polar equation of the sinusoidal spiral to the alternative form $r^n=\sin(n\theta)$.

The sinusoidal spiral $r^n=\cos(n\theta)$, when $n$ is a positive integer, can be identified with the set of complex
numbers $z=re^{i\theta}$ satisfying the equation $\left|z^n-\frac{1}{2}\right|=\frac{1}{2}$. Thus it is a special case
of a general family of planar curves called \textit{polynomial lemniscates} \cite{Wikipedia:polynomial_lemniscate},
which take the form
\[\left\{z\in\mathbf{C}\ |\ |p(z)|=c\right\},\]
where $p(z)$ is a complex polynomial and $c$ is a positive real number.

In the context of this paper, sinusoidal spirals are of particular interest because their arc length formulas are given by
integrals which are simple generalizations of the elliptic integral $\int\frac{\,dr}{\sqrt{1-r^4}}$ giving the arc length of
the lemniscate. Recall that the differential arc length formula for a polar curve is
\begin{equation}\label{infinitesimal_arc_length}
ds=\sqrt{r^2d\theta^2 +dr^2}.
\end{equation}
Usually the polar equation is given in the form $r=f(\theta)$, in which case the arc length integral takes the form
\begin{equation}
L=\int_{\Theta_1}^{\Theta_2}ds =\int_{\Theta_1}^{\Theta_2}\sqrt{r^2+\left(\frac{dr}{d\theta}\right)^2}\,d\theta.
\end{equation}
However if the polar equation can be written in the form $\theta=g(r)$, then the arc length formula takes the form
\begin{equation}\label{arc_length_in_r}
L=\int_{R_1}^{R_2}ds =\int_{R_1}^{R_2}\sqrt{1+r^2\left(\frac{d\theta}{dr}\right)^2}\,dr.
\end{equation}
In the case of the sinusoidal spiral restricted to the interval $0\le\theta\le\frac{\pi}{2n}$, the polar equation 
can be rewritten as $\theta=\frac{1}{n}\arccos(r^n)$, and using equation \eqref{arc_length_in_r} we obtain
\begin{align}
L &\quad=\quad \int_{R_1}^{R_2}\sqrt{1+r^2\left(-\frac{r^{n-1}}{\sqrt{1-r^{2n}}}\right)^2}\,dr\nonumber\\
&\quad=\quad \int_{R_1}^{R_2}\sqrt{1+\frac{r^{2n}}{1-r^{2n}}}\,dr\nonumber\\
L&\quad=\quad \int_{R_1}^{R_2}\frac{\,dr}{\sqrt{1-r^{2n}}},\label{sinusoidal_arc_length}
\end{align}
where $0\le R_1\le R_2\le 1$.

For reference we also recall the formula for the area of radial sectors of polar curves:
\begin{equation}\label{polar_area_formula}
A=\int_{\Theta_1}^{\Theta_2}\frac{1}{2}r^2\,d\theta.
\end{equation}

\bigskip

\section{The Lam\'e-Sinusoidal Spiral Relation}

We first establish the fundamental integral identity linking the geometry of the Lam\'e curve to that of the sinusoidal spiral.

\begin{theorem}\label{fundamental_integral_theorem}
For any positive integer $n$, the following equality holds:
\begin{equation}\label{fundamental_identity}
    \int_0^1 \frac{\,dr}{\sqrt{1-r^{2n}}} = 2^{1/n} \int_0^1 \sqrt[2n]{1-x^{2n}} \, dx
\end{equation}
\end{theorem}

\begin{proof}
The integral $\calA=\int_0^1 \sqrt[2n]{1-x^{2n}} \, dx$ represents the area of the portion of
the Lam\'e curve $x^{2n}+y^{2n}=1$ contained within the first quadrant. An alternative way of
computing this area is to use the following parametrization of the Lam\'e curve:
\[x=\cos^{\frac{1}{n}}(nt),\ y=\sin^{\frac{1}{n}}(nt),\quad 0\le t\le\frac{\pi}{2n}.\]
(Note that these are the same as the polar equations of the sinusoidal spiral restricted to the interval $\left[0,\frac{\pi}{2n}\right]$.)
Thus by Green's theorem, we have
\begin{align*}
\calA &=\frac{1}{2}\int_{\partial\{(x,y)\,|\,x^{2n}+y^{2n}\le1,\,x\ge0,\,y\ge0\}}(x\,dy-y\,dx)\\
&=0+0+\frac{1}{2}\int_0^{\frac{\pi}{2n}}\left(x\frac{dy}{dt}-y\frac{dx}{dt}\right)\,dt\\
&=\frac{1}{2}\int_0^{\frac{\pi}{2n}}\left(\frac{\cos^{\frac{1}{n}}(nt)\sin^{\frac{1}{n}}(nt)\cos(nt)}{\sin(nt)}+
\frac{\sin^{\frac{1}{n}}(nt)\cos^{\frac{1}{n}}(nt)\sin(nt)}{\cos(nt)}\right)\,dt\\
&=\frac{1}{2}\int_0^{\frac{\pi}{2n}}\left(\frac{\cos^{\frac{1}{n}}(nt)\sin^{\frac{1}{n}}(nt)\left(\cos^2(nt)+\sin^2(nt)\right)}{\cos(nt)\sin(nt)}\right)\,dt\\
&=\frac{1}{2^{\frac{1}{n}}}\int_0^{\frac{\pi}{2n}}\left(\frac{2^{\frac{1}{n}}\cos^{\frac{1}{n}}(nt)\sin^{\frac{1}{n}}(nt)}{2\cos(nt)\sin(nt)}\right)\,dt\\
&=\frac{1}{2^{\frac{1}{n}}}\int_0^{\frac{\pi}{2n}}\frac{\sin^{\frac{1}{n}}(2nt)}{\sin(2nt)}\,dt
\end{align*}
Substituting $u=2t$ into the last integral, we obtain
\begin{align}
\calA&=\frac{1}{2^{1+\frac{1}{n}}}\int_0^{\frac{\pi}{n}}\frac{\sin^{\frac{1}{n}}(nu)}{\sin(nu)}\,du\nonumber\\
&=
\frac{1}{2^{1+\frac{1}{n}}}\left[\int_0^{\frac{\pi}{2n}}\frac{\sin^{\frac{1}{n}}(nu)}{\sin(nu)}\,du+\int_{\frac{\pi}{2n}}^{\frac{\pi}{n}}\frac{\sin^{\frac{1}{n}}(nu)}{\sin(nu)}\,du\right]\label{area_eqn2}
\end{align}
Using the substitution $v=\frac{\pi}{n}-u$, we see that
\begin{align*}
\int_{\frac{\pi}{2n}}^{\frac{\pi}{n}}\frac{\sin^{\frac{1}{n}}(nu)}{\sin(nu)}\,du &=-\int_{\frac{\pi}{2n}}^0\frac{\sin^{\frac{1}{n}}(\pi-nv)}{\sin(\pi-nv)}\,dv\\
&=\int_0^{\frac{\pi}{2n}}\frac{\sin^{\frac{1}{n}}(nv)}{\sin(nv)}\,dv\\
&=\int_0^{\frac{\pi}{2n}}\frac{\sin^{\frac{1}{n}}(nu)}{\sin(nu)}\,du
\end{align*}
Thus \eqref{area_eqn2} simplifies to
\begin{equation}
\calA=\frac{1}{2^{1+\frac{1}{n}}}\left[2\int_0^{\frac{\pi}{2n}}\frac{\sin^{\frac{1}{n}}(nu)}{\sin(nu)}\,du\right]=
\frac{1}{2^{\frac{1}{n}}}\int_0^{\frac{\pi}{2n}}\frac{\sin^{\frac{1}{n}}(nu)}{\sin(nu)}\,du.
\end{equation}
Making the substitution $r=\sin^{\frac{1}{n}}(nu)$, we obtain
\[\frac{dr}{du}=\frac{\sin^{\frac{1}{n}}(nu)\cos(nu)}{\sin(nu)}=\frac{\sin^{\frac{1}{n}}(nu)\sqrt{1-r^{2n}}}{\sin(nu)},\]
which implies that
\begin{equation}
\calA=\frac{1}{2^{\frac{1}{n}}}\int_0^1\frac{\,dr}{\sqrt{1-r^{2n}}}
\end{equation}
Hence
\[\int_0^1\frac{\,dr}{\sqrt{1-r^{2n}}}=2^{\frac{1}{n}}\calA=2^{\frac{1}{n}}\int_0^1 \sqrt[2n]{1-x^{2n}} \, dx.\]
This completes the proof.
\end{proof}

\bigskip

\begin{remark}
A more concise proof of Theorem \ref{fundamental_integral_theorem} can be carried out using the same approach
used in \cite{alevin} and \cite[p. 173]{squigonometry_book} to prove \eqref{alevin_eqn}. Namely evaluating the two
integrals on both sides of \eqref{fundamental_identity} in terms of the beta ${B(x,y)=\int_0^1 t^{x-1}(1-t)^{y-1}\,dt}$
and gamma ${\Gamma(x)=\int_0^\infty t^{x-1}e^{-t}\,dt}$ functions and simplifying (\cite{eartin}, \cite{whittaker-watson},
\cite{Wikipedia:beta_function}, \cite{Wikipedia:gamma_function}) yields
\begin{align*}
\int_0^1 \sqrt[2n]{1-x^{2n}} \, dx &= \frac{1}{2n}B\left(\frac{1}{2n},1+\frac{1}{2n}\right)
=\frac{\Gamma\left(\frac{1}{2n}\right)^2}{4n\Gamma\left(\frac{1}{n}\right)}\\
\int_0^1\frac{\,dr}{\sqrt{1-r^{2n}}} &=\frac{1}{2n}B\left(\frac{1}{2n},\frac{1}{2}\right)
=\frac{\sqrt{\pi}\Gamma\left(\frac{1}{2n}\right)}{2n\Gamma\left(\frac{1}{2n}+\frac{1}{2}\right)}
=\frac{\Gamma\left(\frac{1}{2n}\right)^2}{2^{2-\frac{1}{n}}n\Gamma\left(\frac{1}{n}\right)}
\end{align*}
Comparison of the two equalities verifies equation \eqref{fundamental_identity}.
\end{remark}

\bigskip

\begin{remark}
In view of \eqref{sinusoidal_arc_length}, the left hand side of equation \eqref{fundamental_identity} is half the
arc length of a leaf of the sinusoidal spiral $r^n=\cos(n\theta)$. If we adopt the notation of \cite{thyde}, \cite{cox}, then
\begin{equation}
\varpi_{2n}=2\int_0^1\frac{\,dr}{\sqrt{1-r^{2n}}}
\end{equation}
is the arc length of a full leaf, and thus the perimeter of the sinusoidal spiral $r^n=\cos(n\theta)$ is $n\varpi_{2n}$.

As noted in the proof, the integral on the right hand side of equation \eqref{fundamental_identity} is  the area of the portion of
the Lam\'e curve $x^{2n}+y^{2n}=1$ contained within the first quadrant. Thus the total area enclosed by the
Lam\'e curve is
\begin{align}
\calA\left(x^{2n}+y^{2n}\le1\right) &\quad=\quad 4\int_0^1\sqrt[2n]{1-x^{2n}}\,dx\nonumber\\
&\quad=\quad 2^{1-\frac{1}{n}}\varpi_{2n}
\end{align}
We can easily extend this to compute the area of the \textit{superellipse} \cite{Wikipedia:superellipse}
${\left(\frac{x}{a}\right)^{2n}+\left(\frac{y}{b}\right)^{2n}=1}$ as follows:
\begin{align}
\calA\left(\left(\frac{x}{a}\right)^{2n}+\left(\frac{y}{b}\right)^{2n}\le1\right)
&\quad=\quad 4\int_0^a b\sqrt[2n]{1-\left(\frac{x}{a}\right)^{2n}}\,dx\nonumber\\
&\quad=\quad 4b\int_0^1 \sqrt[2n]{1-u^{2n}}\,a\,du\quad\left(u=\frac{x}{a}\right)\nonumber\\
&\quad=\quad 4ab\int_0^1 \sqrt[2n]{1-u^{2n}}\,du\quad\nonumber\\
&\quad=\quad 2^{1-\frac{1}{n}}\varpi_{2n}ab
\end{align}
Moreover it is clear that the proof of Theorem \ref{fundamental_integral_theorem} and the above calculations go
forward when the exponent $2n$ is replaced by an arbitrary positive real number $\alpha$. Thus we obtain the
following corollary.
\end{remark}

\bigskip

\begin{corollary}
Let $\alpha>0$. Then the area enclosed by the superellipse 
${\left|\frac{x}{a}\right|^{\alpha}+\left|\frac{y}{b}\right|^{\alpha}=1}$ is
\begin{equation}
\calA=2^{1-\frac{2}{\alpha}}\varpi_{\alpha}ab,
\end{equation}
where ${\varpi_{\alpha}=2\int_0^1\frac{\,dr}{\sqrt{1-r^{\alpha}}}}$ represents the arc length of the principal leaf of the sinusoidal
spiral ${r^{\frac{\alpha}{2}}=\cos\left(\frac{\alpha}{2}\,\theta\right)}$ within the polar sector 
${-\frac{\pi}{\alpha}\le\theta\le\frac{\pi}{\alpha}}$.
\end{corollary}

\bigskip


\section{Generalized Geometric Relations}

In our previous work, we utilized a result of Siegel to establish a geometric correspondence between the lemniscate and the squircle,
\cite[Theorem 10]{Fiedorowicz-Ramalingam}. We now provide a generalized version of this theorem which relates the sinusoidal spirals to the generalized Lam\'e curves.

\begin{theorem}\label{siegel_relation}
Let $n$ be a positive real number, $0\le T\le1$, and $R^n=\frac{2T^n}{1+T^{2n}}$. Then
\begin{equation}\label{siegel_eqn}
\int_0^R\frac{\,dr}{\sqrt{1-r^{2n}}}=2^{\frac{1}{n}}\int_0^T\frac{\,dv}{\sqrt[n]{1+v^{2n}}}.
\end{equation}
\end{theorem}

\begin{proof}
We make the substitution $r^n=\frac{2v^n}{1+v^{2n}}$ in the integral on the left hand side. Then we have
\[\left(\frac{r}{v}\right)^n = \frac{2}{1+v^{2n}}\]
Raising both sides of this equation to the $-\frac{n-1}{n}$ power, we obtain
\begin{equation}
\left(\frac{v}{r}\right)^{n-1}=\frac{(1+v^{2n})^{1-\frac{1}{n}}}{2^{1-\frac{1}{n}}}
\end{equation}
We also have
\begin{align*}
1-r^{2n} &\quad=\quad 1-\left(\frac{2v^n}{1+v^{2n}}\right)^2\\
&\quad=\quad\frac{(1+v^{2n})^2-4v^{2n}}{(1+v^{2n})^2}\\
&\quad=\quad\frac{1+2v^{2n}+v^{4n}-4v^{2n}}{(1+v^{2n})^2}\\
&\quad=\quad\frac{1-2v^{2n}+v^{4n}}{(1+v^{2n})^2}\\
&\quad=\quad\frac{(1-v^{2n})^2}{(1+v^{2n})^2}
\end{align*}
Hence we have
\begin{equation}
\frac{1}{\sqrt{1-r^{2n}}}=\frac{1+v^{2n}}{1-v^{2n}}
\end{equation}
Differentiating the substitution we obtain
\begin{align*}
nr^{n-1}\,\frac{dr}{dv} &\quad=\quad\frac{2nv^{n-1}(1+v^{2n})-2v^n\cdot2nv^{2n-1}}{(1+v^{2n})^2}\\
&\quad=\quad\frac{2nv^{n-1}+2nv^{3n-1}-4nvv^{3n-1}}{(1+v^{2n})^2}\\
&\quad=\quad\frac{2nv^{n-1}-2nv^{3n-1}}{(1+v^{2n})^2}\\
&\quad=\quad 2nv^{n-1}\frac{1-v^{2n}}{(1+v^{2n})^2}
\end{align*}
Hence we have
\begin{equation}
dr= 2\left(\frac{v}{r}\right)^{n-1}\frac{1-v^{2n}}{(1+v^{2n})^2}\,dv
\end{equation}
Substituting (11), (12) and (13) into the integral on the left hand side of (10), we obtain
\begin{align*}
\int_0^R\frac{\,dr}{\sqrt{1-r^{2n}}} &\quad=\quad\int_0^T\left(\frac{1+v^{2n}}{1-v^{2n}}\right)\frac{1}{2}\left(\frac{v}{r}\right)^{n-1}\frac{1-v^{2n}}{(1+v^{2n})^2}\,dv\\
&\quad=\quad\frac{1}{2}\int_0^T\left(\frac{1}{1+v^{2n}}\right)\left(\frac{(1+v^{2n})^{1-\frac{1}{n}}}{2^{1-\frac{1}{n}}}\right)\,dv\\
&\quad=\quad 2^{\frac{1}{n}}\int_0^T\frac{\,dv}{\sqrt[n]{1+v^{2n}}}.
\end{align*}
\end{proof}
\bigskip

Theorem \ref{siegel_relation} has the following geometric interpretation as illustrated by the following figure in the case $n=3$.

\begin{figure}[H]
\centering
\includegraphics[width=300pt]{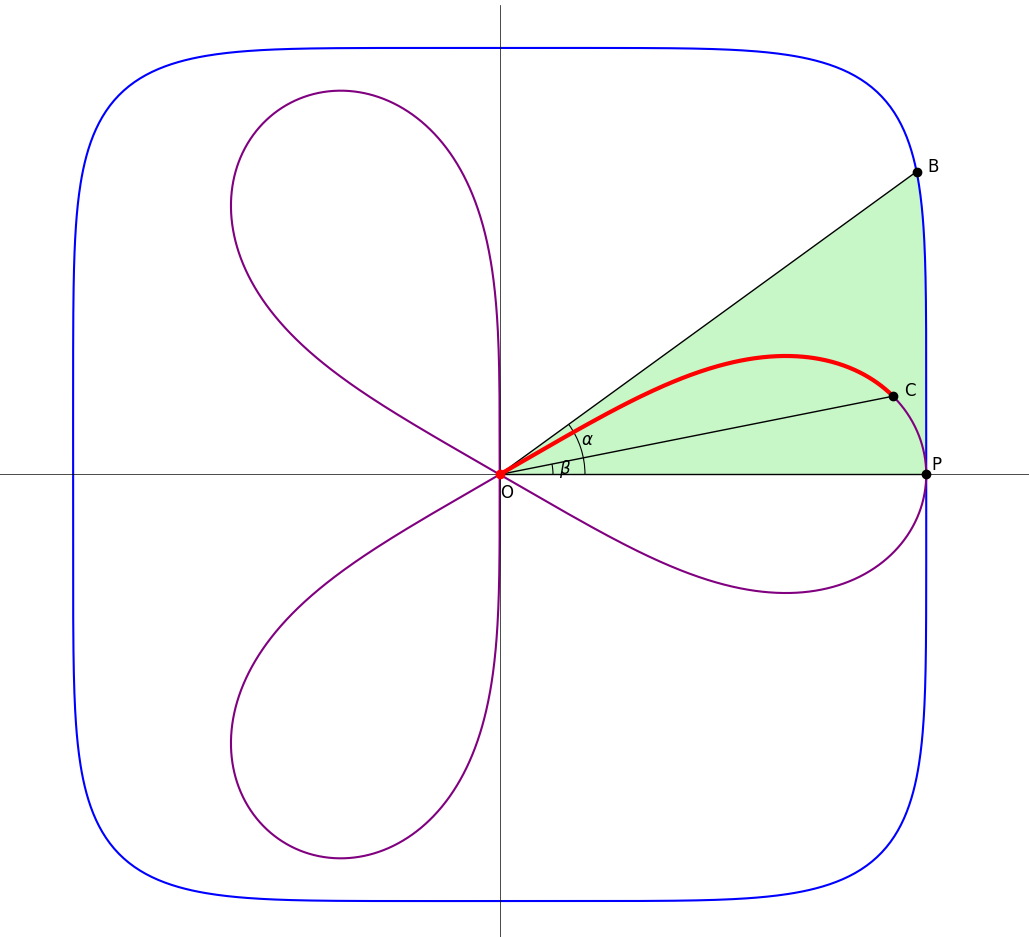}
\caption{Relation Between $x^6+y^6=1$  and $r^3=\cos(3\theta)$}
\label{fig:siegel_lame_spiral}
\end{figure}

\bigskip

\begin{corollary}\label{siegel_corollary}
Let $0\le\alpha\le\frac{\pi}{4}$, $T=\tan(\alpha)$, and $\beta=\frac{1}{n}\arccos\left(\frac{2T^n}{1+T^{2n}}\right)$.
Then
\begin{equation}
l=2^{1+\frac{1}{n}}a,
\end{equation}
where $l$ is the length of the arc of the sinusoidal spiral $r^n=\cos(n\theta)$ within the polar sector $\beta\le\theta\le\frac{\pi}{2n}$ and $a$ is the area of the radial sector of the Lam\'e curve $x^{2n}+y^{2n}=1$ within
$0\le\theta\le\alpha$.
\end{corollary}

\begin{proof}
By equation \eqref{sinusoidal_arc_length}
\begin{equation}\label{length_eqn}
l=\int_0^R\frac{\,dr}{\sqrt{1-r^{2n}}},
\end{equation}
where
\[R^n=\cos(n\beta)=\frac{2T^n}{1+T^{2n}}.\]

On the other hand the cartesian equation of the Lam\'e curve $x^{2n}+y^{2n}=1$ converts to polar form as follows:
\begin{align*}
[r \cos(\theta)]^{2n} + [r \sin(\theta)]^{2n} &\quad=\quad 1\\
r^{2n} (\cos^{2n}(\theta) + \sin^{2n}(\theta) &\quad=\quad 1\\
r^{2n}\cos^{2n}(\theta)\left[1+\tan^{2n}(\theta)\right]&\quad=\quad 1,
\end{align*}
which yields
\begin{align*}
r^{2n}&\quad=\quad\frac{\sec^{2n}(\theta)}{1+\tan^{2n}(\theta)}\\
r^2&\quad=\quad\frac{\sec^{2}(\theta)}{\sqrt[n]{1+\tan^{2n}(\theta)}}.
\end{align*}
Then by \eqref{polar_area_formula}
\begin{align}
a = \int_0^\alpha\frac{1}{2}r^2\,d\theta &=
 \frac{1}{2}\int_0^\alpha \frac{\sec^{2}(\theta)}{\sqrt[n]{1+\tan^{2n}(\theta)}}\nonumber\\
&= \frac{1}{2}\int_0^T \frac{\,dv}{\sqrt[n]{1+v^{2n}}}\quad(v=\tan(\theta))\label{area_eqn}
\end{align}
Comparison of equations \eqref{siegel_eqn}, \eqref{length_eqn} and \eqref{area_eqn} establishes the relation
$l=2^{1+\frac{1}{n}}a$.
\end{proof}

\bigskip

\begin{remark}
In the case $n=2$, there is a simple relation between the radial polar coordinates of the point $B$ on the squircle
$x^4+y^4=1$ and the point $C$ on the lemniscate $r^2=\cos(2\theta)$ whose angular polar coordinates are
$\theta=\alpha$ and $\theta=\beta$ respectively. Namely $R=\overline{OC}$ and $S=\overline{OB}$ are related by.
\begin{align*}
R^2 &=\cos(2\beta)=\frac{2T^2}{1+T^4}\\
S^2 &=\frac{\sec^2(\alpha)}{\sqrt{1+\tan^4(\alpha)}}=\frac{1+T^2}{\sqrt{1+T^4}}
\end{align*}
Hence
\[S^4=\frac{1+2T^2+T^4}{1+T^4}=1+\frac{2T^2}{1+T^4}=1+R^2.\]
\end{remark}

\vspace{30pt}

\begin{remark}
Theorem \ref{fundamental_integral_theorem}  is a consequence of Corollary \ref{siegel_corollary}. For if we take $\alpha=\frac{\pi}{4}$, then $T=1$,
$R=1$ and $\beta=0$. Thus $l=\int_0^1\frac{\,dr}{\sqrt{1-r^{2n}}}$ and $a=\frac{1}{2}\int_0^1\frac{\,dv}{\sqrt[n]{1+v^{2n}}}$ is $\frac{1}{8}$ of the area
enclosed by the Lam\'e curve $x^{2n}+y^{2n}=1$. Since $\int_0^1\sqrt[2n]{1-x^{2n}}\,dx$ is $\frac{1}{4}$ of this area, we have
\begin{equation}
\int_0^1\frac{\,dr}{\sqrt{1-r^{2n}}}=l=2^{1+\frac{1}{n}}a=2^{\frac{1}{n}}(2a)=2^{\frac{1}{n}}\int_0^1\sqrt[2n]{1-x^{2n}}\,dx
\end{equation}
\end{remark}

\bigskip

\section{Physical Interpretation: Keplerian vs Uniform Motion}\label{physics_interpretation}

The relation described in Corollary \ref{siegel_corollary}  implies that Keplerian motion along the Lam\'e curve $x^{2n}+y^{2n}=1$ within the polar sector  $0\le\theta\le\frac{\pi}{4}$ corresponds geometrically to uniform motion along the sinusoidal spiral $r^n=\cos(n\theta)$  within the polar sector $0\le\theta\le\frac{\pi}{2n}$.  By symmetry this correspondence extends to such a correspondence between the entire Lam\'e curve and the entire sinusoidal spiral, answering a question posed in \cite[Remark 6]{Fiedorowicz-Ramalingam}.

We illustrate this relation graphically for the case $n=3$ in Figure \ref{fig:Keplerian}. The area sectors on the Lam\'e curve (swept at a constant areal rate) correspond precisely to the arc lengths on the sinusoidal spiral (traversed at a constant speed).

\begin{figure}[H]
\centering
\includegraphics[width=\textwidth]{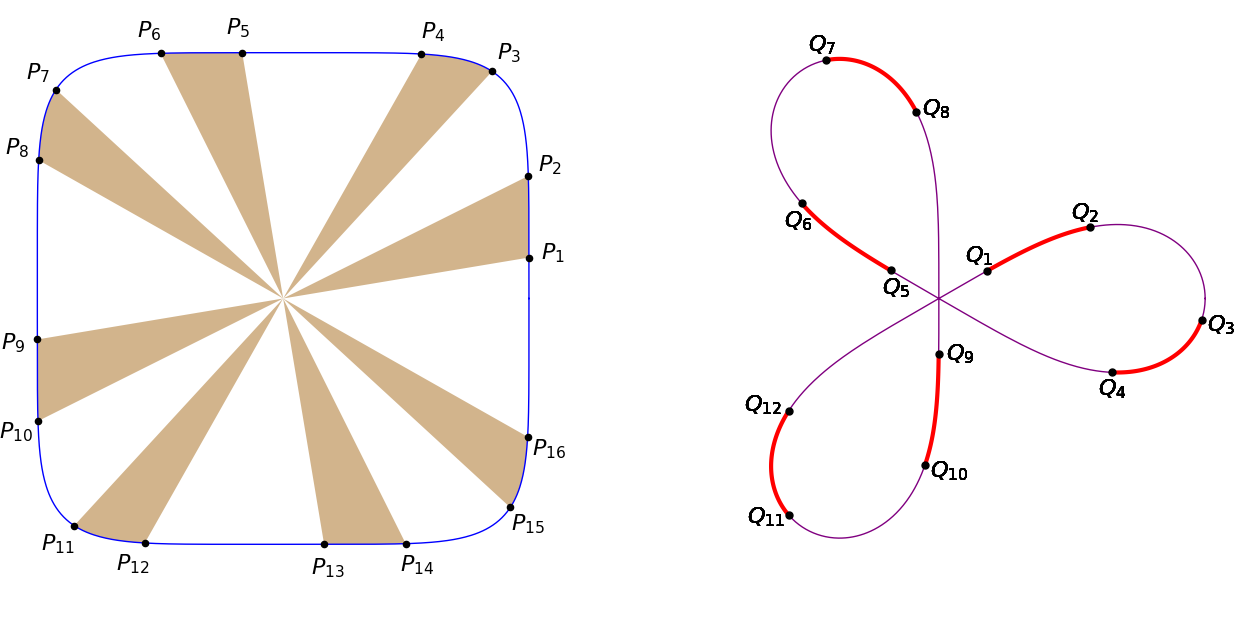}
\caption{Comparison of Keplerian motion on the Lam\'e curve $x^6+y^6=1$ and uniform motion on the sinusoidal spiral $r^3=\cos(3\theta)$.}
\label{fig:Keplerian}
\end{figure}
A full cycle of paired motions on the Lam\'e curve and the sinusoidal spiral is described by the following pattern:
\begin{align*}
\to(P_1,Q_1)&\to(P_2,Q_2)\to(P_3,Q_3)\to(P_4,Q_4)\to(P_5,Q_5)\to(P_6,Q_6)\\
&\to(P_7,Q_1)\to(P_8,Q_2)\to(P_1,Q_3)\to(P_2,Q_4)\to(P_3,Q_5)\to(P_4,Q_6)\\
&\to(P_5,Q_1)\to(P_6,Q_2)\to(P_7,Q_3)\to(P_8,Q_4)\to(P_1,Q_5)\to(P_2,Q_6)\\
&\to(P_3,Q_1)\to(P_4,Q_2)\to(P_5,Q_3)\to(P_6,Q_4)\to(P_7,Q_5)\to(P_8,Q_6)\to
\end{align*}

\section{General Force Law for Lam\'e Curves}

It is well known that a particle moving in Keplerian motion around the origin is subject to a central force directed
towards the origin. In this section we derive an explicit formula for the central force which compels a particle
to move in Keplerian motion around the Lam\'e curve $x^{2n}+y^{2n}=1$.

\begin{theorem}
For $n>1$ the central force law for a particle in Keplerian motion around the Lam\'e curve $x^{2n}+y^{2n}=1$ has the form
\begin{equation}\label{central_force_formula}
F(r)=-Cr^{4n-3}w^{2n-2},
\end{equation}
where $w = \sin\theta\cos\theta = \frac{1}{2}\sin(2\theta)$ and $C$ is a constant depending on the physical properties
of the moving particle.
\end{theorem}

\begin{proof}
The differential equation of an orbit under a central force $F(r)$ directed towards the origin is given by Binet's equation \cite{Wikipedia:binet_equation}, \cite[p. 87 (3.34)]{goldstein}, \cite{cline}, \cite{cooper}:
\begin{equation} \label{binet}
    F(r) = -m h^2 u^2 \left( u + \frac{d^2u}{d\theta^2} \right)
\end{equation}
where:
\begin{itemize}
    \item $u=\frac{1}{r}$,
    \item $h$ is the specific angular momentum (constant),
    \item $m$ is the mass of the orbiting body.
\end{itemize}

We begin by rewriting the cartesian equation for the Lam\'e curve in polar coordinates:
 $x = r \cos(\theta)$ and $y = r \sin(\theta)$:
\begin{align*}
x^{2n}+y^{2n} &\quad=\quad 1\\
[r \cos(\theta)]^{2n} + [r \sin(\theta)]^{2n} &\quad=\quad 1,
\end{align*}
obtaining
\begin{align*}
r^{2n} (\cos^{2n}(\theta) + \sin^{2n}(\theta) &\quad=\quad 1\\
r^{2n} &\quad=\quad \frac{1}{\cos^{2n}(\theta) + \sin^{2n}(\theta)}\\
u^{2n} &\quad=\quad \cos^{2n} (\theta) + \sin^{2n} (\theta)
\end{align*}

To avoid notational clutter we abbreviate $c=\cos(\theta)$ and $s=\sin(\theta)$. Thus
\begin{equation}
u^n=c^{2n}+s^{2n}:=A\label{A_defn}
\end{equation}
Differentiating with respect to $\theta$:
\begin{align*}
2n u^{2n-1} u' &= 2n s^{2n-1} c - 2n c^{2n-1} s \\
u^{2n-1} u' &= sc(s^{2n-2} - c^{2n-2})
\end{align*}
Let $w = sc$ and $B = s^{2n-2} - c^{2n-2}$.
We have:
\begin{equation}
u^{2n-1} u' = w B \label{eq:first_deriv}
\end{equation}

Differentiating (\ref{eq:first_deriv}) with respect to $\theta$ we obtain:
\begin{equation}
(2n-1)u^{2n-2}(u')^2 + u^{2n-1}u'' = (wB)' = w'B + wB'
\end{equation}
We compute the derivatives of the components:
\begin{itemize}
    \item $w' = (sc)' = c^2 - s^2$
    \item $B' = (s^{2n-2} - c^{2n-2})' = (2n-2)s^{2n-3}c - (2n-2)c^{2n-3}(-s)$
    \item $B' = (2n-2)sc(s^{2n-4} + c^{2n-4}) = (2n-2)w C$, where $C = s^{2n-4} + c^{2n-4}$.
\end{itemize}
Substituting these back we get:
\begin{equation}
(2n-1)u^{2n-2}(u')^2 + u^{2n-1}u'' = (c^2-s^2)B + (2n-2)w^2 C \label{eq:second_deriv}
\end{equation}

It will be convenient in what follows to take 
\[X = u^{4n-1}(u+u'') = u^{2n}(u^{2n-1}u'') + u^{4n}=A(u^{2n-1}u'').\]
From (\ref{eq:second_deriv}), we isolate $u^{2n-1}u''$:
$$ u^{2n-1}u'' = (c^2-s^2)B + (2n-2)w^2 C - (2n-1)u^{2n-2}(u')^2 $$
From (\ref{eq:first_deriv}), we have $(u')^2 = \frac{w^2 B^2}{u^{4n-2}}$.
Thus:
$$ (2n-1)u^{2n-2}(u')^2 = \frac{(2n-1)w^2 B^2}{u^{2n}} $$
Substituting this into the expression for $X$:
\begin{align*}
X &= A \left[ (c^2-s^2)B + (2n-2)w^2 C - \frac{(2n-1)w^2 B^2}{A} \right] + A^2 \\
X &= A^2+A(c^2-s^2)B + (2n-2)w^2 AC - (2n-1)w^2 B^2
\end{align*}
We simplify the first two terms $A^2+A(c^2-s^2)B$:
\begin{align*}
(c^2-s^2)B &= (c^2-s^2)(s^{2n-2}-c^{2n-2}) \\
&= s^{2n-2}c^2 - c^{2n} - s^{2n} + s^2c^{2n-2} \\
&= s^2c^2(s^{2n-4} + c^{2n-4}) - (s^{2n} + c^{2n}) \\
&= w^2 C - A
\end{align*}
Therefore:
\[A(c^2-s^2)B + A^2 = A(w^2 C - A) + A^2 = A w^2 C - A^2 + A^2 = A w^2 C.\]
Substituting this simplified term back into $X$:
\begin{align*}
X &= (A w^2 C) + (2n-2)A w^2 C - (2n-1)w^2 B^2 \\
X &= (2n-1) A w^2 C - (2n-1) w^2 B^2 \\
X &= (2n-1) w^2 [ AC - B^2 ]
\end{align*}

Finally we evaluate $AC - B^2$:
\begin{align*}
AC &= (s^{2n}+c^{2n})(s^{2n-4}+c^{2n-4}) = s^{4n-4} + c^{4n-4} + s^{2n}c^{2n-4} + c^{2n}s^{2n-4} \\
B^2 &= (s^{2n-2}-c^{2n-2})^2 = s^{4n-4} + c^{4n-4} - 2s^{2n-2}c^{2n-2}
\end{align*}
Subtracting:
\begin{align*}
AC - B^2 &= s^{2n}c^{2n-4} + c^{2n}s^{2n-4} + 2s^{2n-2}c^{2n-2} \\
&= s^{2n-4}c^{2n-4} ( s^4 + c^4 + 2s^2c^2 ) \\
&= w^{2n-4} (s^2+c^2)^2 \\
&= w^{2n-4}(1)^2=w^{2n-4}
\end{align*}
Substituting back into $X$ we get
$$ X = (2n-1) w^2 [ w^{2n-4} ] = (2n-1) w^{2n-2} $$
Since $X = u^{4n-1}(u+u'')$, we obtain the following:
\begin{equation}
u + u'' = \frac{(2n-1)w^{2n-2}}{u^{4n-1}}\label{u+ddu}
\end{equation}

Substituting (\ref{u+ddu}) into equation \eqref{binet} we obtain
\begin{align*}
F(r) &\quad=\quad -m h^2 u^2 \left( \frac{(2n-1)w^{2n-2}}{u^{4n-1}}\right)\\
&\quad=\quad -C\left( \frac{(2n-1)w^{2n-2}}{u^{4n-3}}\right)\\
&\quad=\quad -Cr^{4n-3}w^{2n-2},
\end{align*}
where $C=(2n-1)mh^2$.
\end{proof}

\bigskip

\begin{remark}
The case $n=1$ is degenerate. A circular orbit is compatible with any attractive central force.
\end{remark}

\bigskip

\begin{remark}
It might appear that equation \eqref{central_force_formula} is a defective central force formula, since it involves both
polar coordinates $r$ and $\theta$ instead of $r$ alone. However one should keep in mind that the polar coordinates
$r$ and $\theta$ are not independent, since they are related to each other by the equation of the Lam\'e curve. In
fact by applying the half-angle formulas we can rewrite
\begin{equation}\label{half-angle-simplification}
u^{2n}=\cos^{2n}(\theta)+\sin^{2n}(\theta)=\left[\frac{1+\cos(2\theta)}{2}\right]^n+\left[\frac{1-\cos(2\theta)}{2}\right]^n
\end{equation}
Moreover in equation \eqref{half-angle-simplification}, the odd powers of $\cos(2\theta)$ cancel out and by another
application of the half-angle formula $\cos^2(2\theta)=\frac{1+\cos(4\theta)}{2}$, $u^{2n}$ can be expressed as a
polynomial in $\cos(4\theta)$. Further using the trigonometric identity 
\[\cos(4\theta)=1-2\sin^2(2\theta)=1-8w^2,\] 
we can rewrite $u^{2n}$ as a polynomial  in $w^2$. For $n=2,3$ the relation between $u^{2n}$ and $w^2$ is linear and for
$n=4,5$ the relation is quadratic. Hence in those cases we can explicitly solve for $w^2$ in terms of $u=\frac{1}{r}$, and
thus obtain central force formulas explicitly in terms of $r$ alone. One obtains the following formulas:
\begin{align*}
n=2 &\quad:\quad F(r)=Cr(1-r^4)\\
n=3 &\quad:\quad F(r)=-C\frac{(1-r^6)^2}{r^3}\\
n=4 &\quad:\quad F(r)=Cr\left(\sqrt{2r^8+2}-2r^4\right)^3\\
n=5 &\quad:\quad F(r)=Cr^2\left(\sqrt{5r^{10}+2}-5r^5\right)^3
\end{align*}
\end{remark}

\bigskip

\section{Sinusoidal Spirals and Policles}

The reader might observe that Theorem 2 generalizes \cite[Theorem 10]{Fiedorowicz-Ramalingam}, which is a
reformulation of the main result \cite[Theorem 1]{Fiedorowicz-Ramalingam}, based on a calculation of Siegel \cite{siegel}. 
The reader might inquire whether there is a simple direct generalization of that main result relating
the squircle and lemniscate to one between Lam\'e curves and sinusoidal spirals. Regrettably, this is
not the case. However there is such a simple direct generalization if we replace the Lam\'e curves $x^{2n}+y^{2n}=1$
by a different generalization of the squircle: curves given by polar equations
\begin{equation}
r^4=\frac{n\sin^2(n\theta)}{1-\cos^{2n}(n\theta)}.
\end{equation}
When $n=2$ this curve is a squircle, a rounded square.
For general $n$ these curves look like rounded regular $2n$-gons.
We call these curves \textit{policles} (``polygon'' $+$ ``circle'') by analogy with squircle (``square'' $+$ ``circle'').
With this modification, we have the following simple direct generalization of \cite[Theorem 1]{Fiedorowicz-Ramalingam},
illustrated by the following figure for the case $n=3$.

\begin{figure}[H]
\centering
\includegraphics[width=250pt]{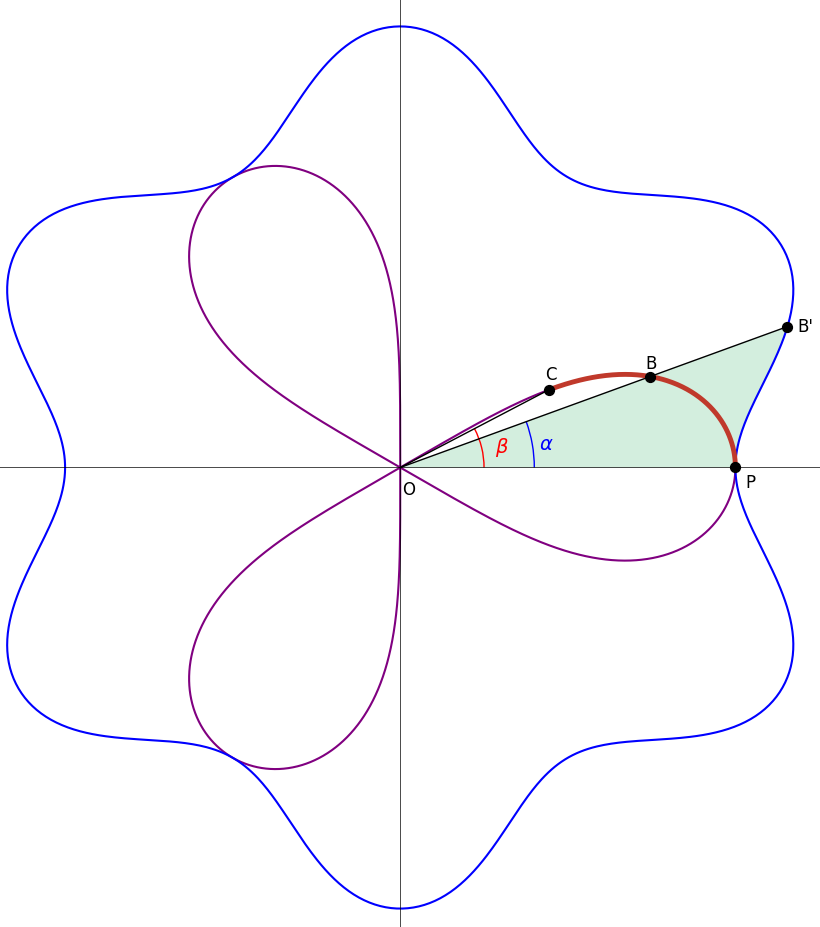}
\caption{Policle and Sinusoidal Spiral.}
\label{fig:policle_spiral}
\end{figure}

\begin{theorem}Let $B$ be a point in the polar sector $0\le\theta\le\frac{\pi}{2n}$ of the sinusoidal spiral
 $r^n=\cos(n\theta)$ and let $B'$ 
be its radial projection onto the policle $r^4=\frac{n\sin^2(n\theta)}{1-\cos^{2n}(n\theta)}$. Let $C$ be the point on this polar sector of the sinusoidal spiral such that 
$\overline{OC}=\overline{OB}^n$. Then
\begin{equation}
l=2a\sqrt{n},
\end{equation}
where $l$ is the arc length of the sinusoidal spiral from $C$ to $P=(1,0)$ and $a$ denotes the area of the policular sector $OPB'$.
\end{theorem}

\begin{proof}
Let the polar coordinates of $B$ be $(r,\theta)=(R_1,\alpha)$ and the polar coordinates of $C$ be 
$(r,\theta)=(R_2,\beta)$. Then by definition we have
\begin{equation}
R_2=\overline{OC}=\overline{OB}^n=R_1^n=\cos(n\alpha)
\end{equation}

Calculating arc lengths and radial sector areas in polar coordinates, we obtain
\begin{equation}
l =\int_{R_2}^1\frac{\,dr}{\sqrt{1-r^{2n}}}=\int_{\cos(n\alpha)}^1\frac{\,dr}{\sqrt{1-r^{2n}}}\label{l_eqn}
\end{equation}
and
\begin{equation}
a=\int_0^{\alpha}\frac{1}{2}r^2\,d\theta=\frac{1}{2}\int_0^{\alpha}\frac{\sqrt{n}\sin(n\theta)d\theta}{\sqrt{1-\cos^{2n}(n\theta)}}\label{a_eqn}
\end{equation}
Making the substitution $r=\cos(n\theta)$ in equation \eqref{a_eqn}, we have
\[dr =-n\sin(n\theta)\,d\theta\quad\Longrightarrow\quad \sin(n\theta)\,d\theta=\frac{1}{n}dr\]
and we obtain
\[a= \frac{1}{2}\int_{\cos(n\alpha)}^1\frac{\sqrt{n}\frac{1}{n}dr}{\sqrt{1-r^{2n}}}=\frac{1}{2\sqrt{n}}\int_{\cos(n\alpha)}^1\frac{\,dr}{\sqrt{1-r^{2n}}}=\frac{1}{2\sqrt{n}}l.\]
Thus $l=2a\sqrt{n}$.
\end{proof}

\bigskip

As in Section \ref{physics_interpretation}, we can interpret this correspondence as relating Keplerian motion around the policle to uniform motion
around the sinusoidal spiral. We leave the details as an exercise for the reader.

\end{document}